\documentclass[12pt]{amsart}
\usepackage[applemac]{inputenc}
\usepackage{bm}
\usepackage[dvipsnames,svgnames,x11names]{xcolor}
\usepackage[hyphens]{url}
\usepackage[colorlinks=true,linkcolor=Maroon,citecolor=blue,urlcolor=blue,hypertexnames=false,linktocpage]{hyperref}

\usepackage{bookmark}
\usepackage{amsmath,thmtools,mathtools}
\mathtoolsset{showonlyrefs=true}
\usepackage{blindtext}
\usepackage{graphicx}

\newcounter{mgncount}

\usepackage{tikz}
\usepackage{fancyhdr}
\usepackage{esint}
\usepackage{enumerate}
\usepackage{pictexwd,dcpic}
\usepackage{graphicx}
\usepackage{caption}
\allowdisplaybreaks
\usepackage{graphicx}

\newtheorem*{conj}{Conjecture}
\newtheorem*{qu-}{Question}
\declaretheorem[name=Theorem,numberwithin=section]{thm}

\declaretheorem[name=Lemma,sibling=thm]{lemma}

\declaretheorem[name=Corollary,sibling=thm]{cor}

\declaretheorem[name=Theorem,numbered=no]{theorem}

\declaretheorem[style=remark,name=Remark,numbered=no]{remark}

\numberwithin{equation}{section}

\newcommand{\wh}{\widehat}

\newcommand{\cn}{\colon}

\newcommand{\bbR}{\mathbb{R}}

\newcommand{\bbS}{\mathbb{S}}

\newcommand{\8}{\infty}
\newcommand{\al}{\alpha}

\newcommand{\de}{\delta}

\newcommand{\la}{\lambda}

\newcommand{\si}{\sigma}
\newcommand{\Si}{\Sigma}

\newcommand{\ze}{\zeta}

\newcommand{\De}{\Delta}

\newcommand{\cC}{\mathcal{C}}
\newcommand{\cD}{\mathcal{D}}

\newcommand{\cK}{\mathcal{K}}

\newcommand{\cP}{\mathcal{P}}

\newcommand{\rt}{\sqrt}
\newcommand{\ip}[2]{\left\langle #1,#2 \right\rangle}

\DeclareMathOperator{\vol}{vol}

\renewcommand{\(}{\left(}
\renewcommand{\)}{\right)}

\newcommand{\eq}[1]{\begin{equation}\begin{alignedat}{2} #1 \end{alignedat}\end{equation}}

\newcommand{\ra}{\rightarrow}

\newcommand{\q}{\quad}
\makeatletter
\@namedef{subjclassname@2020}{%
\textup{2020} Mathematics Subject Classification}
\makeatother

\begin{document}

\title[Capillary Blaschke-Santal\'o Inequality]{On the Conjectured Capillary Blaschke-Santal\'o Inequality}
	\author[C. Cabezas-Moreno, Y. Hu, M. N. Ivaki]{Carlos Cabezas-Moreno, Yingxiang Hu, Mohammad N. Ivaki}
\begin{abstract}
We prove that the conjectured capillary Blaschke--Santal\'{o} inequality holds for any unconditional, strictly convex capillary hypersurface when $\theta \in \left(0, \tfrac{\pi}{2}\right)$. Moreover, for $\theta \in \left(\tfrac{\pi}{2}, \pi\right)$, we show that the capillary volume product has no finite upper bound.
\end{abstract}
\maketitle

\section{Introduction}
Let $(\mathbb{R}^n, \delta = \langle \cdot, \cdot \rangle, D)$ be the Euclidean space equipped with its standard inner product and flat connection. Denote by $(\mathbb{S}^{n-1}, \bar{g}, \bar{\nabla})$ the unit sphere with its induced round metric and Levi-Civita connection. Let $\{E_i\}_{i=1}^n$ be the standard orthonormal basis of $\mathbb{R}^n$, and define the upper half-space as
\[
\mathbb{R}^n_+ = \{ x \in \mathbb{R}^n : \langle x, E_n \rangle > 0 \}.
\] 

Let $\Sigma$ be a properly embedded, smooth, compact, connected, and orientable hypersurface contained in the closure $\overline{\mathbb{R}^n_+}$, such that its relative interior lies entirely in $\mathbb{R}^n_+$ and its relative boundary satisfies $\partial \Sigma \subset \partial \mathbb{R}^n_+$. We say that $\Sigma$ is a capillary (or $\theta$-capillary) hypersurface with constant contact angle $\theta \in (0, \pi)$ if the following condition holds along $\partial \Sigma$:
\begin{equation}\label{capillary angle condition}
\langle \nu, E_n \rangle \equiv \cos \theta,
\end{equation}
where $\nu$ denotes the outer unit normal of $\Sigma$. The region enclosed by the hypersurface $\Sigma$ and the hyperplane $\partial \mathbb{R}^n_+$ is denoted by $\widehat{\Sigma}$. We say that a capillary hypersurface $\Sigma$ is strictly convex if $\widehat{\Sigma}$ is a convex body and the second fundamental form of $\Sigma$ is positive definite.  

The capillary spherical cap of radius $r$ intersecting $\partial \mathbb{R}^n_+$ at a constant angle $\theta \in (0, \pi)$ is defined as
\begin{equation}
\mathcal{C}_{\theta, r} \coloneqq \left\{ x \in \overline{\mathbb{R}^n_+} \mid |x + r \cos \theta E_n| = r \right\}.
\end{equation}
For simplicity, we set $\mathcal{C}_\theta = \mathcal{C}_{\theta, 1}$ and $\mathbb{S}^{n-1}_\theta = \{ x \in \mathbb{S}^{n-1} \mid x_n \geq \cos\theta \}$.

Let $\tilde{\nu} = \nu - \cos\theta\, E_n : \Sigma \to \mathcal{C}_\theta$ denote the capillary Gauss map of $\Sigma$. When $\Sigma$ is strictly convex, $\tilde{\nu}$ is a diffeomorphism (cf. \cite[Lem. 2.2]{MWWX25}). In this case, the capillary support function of $\Sigma$, $s_\Sigma : \mathcal{C}_\theta\to \bbR$, is defined as
\begin{equation}
s_\Sigma(\zeta) = \langle \tilde{\nu}^{-1}(\zeta), \zeta + \cos\theta E_n \rangle, \quad \zeta \in \mathcal{C}_\theta.
\end{equation}

For example, the capillary support function of $\mathcal{C}_\theta$ is given by
\begin{equation}
\ell(\zeta) = \sin^2\theta - \cos\theta \,\langle \zeta, E_n \rangle = 1 - \cos\theta\, \langle x, E_n \rangle
\end{equation}
where $x := \zeta + \cos\theta\, E_n$.

Moreover, the hypersurface $\Sigma$ can be parametrized via the inverse capillary Gauss map:
\begin{equation}
X : \mathcal{C}_\theta \to \Sigma, \q X(\zeta) = \tilde{\nu}^{-1}(\zeta) = \nu^{-1}(\zeta + \cos\theta E_n).
\end{equation}

We say that a function $\varphi \in C(\mathcal{C}_\theta)$ is even if
\begin{equation}
    \varphi(-\zeta_1, \ldots, -\zeta_{n-1}, \zeta_n) 
    = \varphi(\zeta_1, \ldots, \zeta_{n-1}, \zeta_n)
    \quad \forall \zeta \in \mathcal{C}_\theta.
\end{equation}
We also say that a function $\varphi \in C(\mathcal{C}_\theta)$ is unconditional if
\begin{equation}
    \varphi(\pm \zeta_1, \ldots, \pm \zeta_{n-1}, \zeta_n) 
    = \varphi(\zeta_1, \ldots, \zeta_{n-1}, \zeta_n)
    \quad \forall \zeta \in \mathcal{C}_\theta.
\end{equation}
Similarly, a capillary hypersurface $\Sigma$ is said to be even (resp. unconditional) if its capillary support function $s_\Sigma$ is even (resp. unconditional). 

Suppose $\Sigma$ is an even, strictly convex capillary hypersurface. Define 
\eq{
Y : \mathcal{C}_\theta \to \overline{\mathbb{R}^n_+},\q Y(\zeta) = \frac{\ell(\zeta)}{s_\Sigma(\zeta)} \zeta.
}
Since $s_\Sigma>0$ on $\mathcal C_\theta$, $Y$
is a radial graph over $\mathcal C_\theta$. It follows that $\Sigma^* =Y(\mathcal C_\theta)$ is star-shaped with respect to the origin, and thus determines a compact region in $\overline{\bbR^n_+}$ bounded by $Y(\mathcal C_\theta)$ and $\partial\bbR^n_+$:
\eq{
\widehat{\Sigma^*}:=\left\{ r\zeta:\ \zeta\in\mathcal C_\theta,\ 0\le r\le \frac{\ell(\zeta)}{s_\Sigma(\zeta)}\right\}
\subset \overline{\bbR^n_+}.
}
We refer to $\widehat{\Sigma^*}$ as the \emph{capillary polar body} of $\widehat{\Sigma}$. Define
\begin{equation}
\sigma_\Sigma(\zeta) = \frac{s_\Sigma(\zeta)}{\ell(\zeta)} \quad \text{for } \zeta \in \mathcal{C}_\theta.
\end{equation}
By \cite[Prop. 2.9]{MWW25}, we have
\begin{equation}
\vol(\widehat{\Sigma^*}) = \int_{\mathcal{C}_\theta} \frac{1}{\sigma_\Sigma^n} \, dV_{\mathcal{C}_\theta},
\end{equation}
where $dV_{\mathcal{C}_\theta} = \frac{1}{n} \ell \, d\si$ is the cone volume measure of $\widehat{\mathcal{C}_\theta}$ and
$\si$ denotes the spherical Lebesgue measure on $\bbS^{n-1}$. We also have
\begin{equation}
\vol(\widehat{\Sigma}) = \int_{\mathcal{C}_\theta} \frac{\sigma_\Sigma}{G_\Sigma} \, dV_{\mathcal{C}_\theta},
\end{equation}
where $G_\Sigma$ denotes the Gauss curvature of $\Sigma$ (as a function on $\mathcal{C}_\theta$).

\begin{remark}
It is clear that $\mathcal{C}_\theta^* = \mathcal{C}_\theta$. In general, $\Sigma^*$ is not a convex capillary hypersurface.
\end{remark}
We now recall the conjectured Blaschke--Santal\'{o} inequality in the class of even capillary hypersurfaces:

\begin{conj}\cite{MWW25}
Let $\theta \in (0, \frac{\pi}{2})$ and $\Sigma$ be an even, smooth, strictly convex capillary hypersurface. Then
\begin{equation}\label{conj}
\vol(\widehat{\Sigma}) \vol(\widehat{\Sigma^*}) \leq \vol(\widehat{\mathcal{C}_\theta})^2.
\end{equation}
Moreover, equality holds if and only if $\Sigma = \lambda \mathcal{C}_\theta$ for some $\lambda > 0$.
\end{conj}

Note that for $\theta = \tfrac{\pi}{2}$, the inequality \eqref{conj} follows from the classical 
Blaschke--Santal\'o inequality, and equality holds for any half-ellipsoid of the form
\eq{
\sum_{i=1}^na_i^2x_i^2=1,\q  x_n\geq 0,\q \text{where $a_i>0$ for $i=1,\ldots, n$. } 
}

 The classical Blaschke--Santal\'o inequality in the class of origin-symmetric convex bodies (i.e., compact convex sets with non-empty interior such that $x \in K \Rightarrow -x \in K$) states that
\begin{equation}\label{eq:BS}
\vol(K)\vol(K^{\circ}) \leq \vol(B^n)^2,
\end{equation}
where $K^{\circ} = \{y \in \mathbb{R}^n : \langle x, y \rangle \leq 1, \, \forall x \in K\}$ is the (standard) polar body of $K$ and $B^n$ denotes the unit ball of $(\mathbb{R}^n, \delta)$. Moreover, equality holds only for origin-centered ellipsoids. The inequality was proved by Blaschke \cite{Bla45} for $n = 2, 3$ and by Santal\'o \cite{San49} for general $n$. The full characterization of the equality cases was later obtained by Saint-Raymond \cite{SR80}. See also \cite{MP90} and \cite{AAKM04}.

We say a function $f:\bbR^n\to \bbR$ is unconditional if
\eq{
f(\pm x_1,\ldots,\pm x_n)=f(x_1,\ldots,x_n) \quad \forall x\in \bbR^n.
}
Similarly, a convex body $K$ is said to be unconditional if its standard support function $h_K$ is unconditional. 

In this paper, we prove \eqref{conj} for unconditional, strictly convex capillary hypersurfaces.

\begin{thm}\label{thm 1}
Let $\theta \in \left(0, \frac{\pi}{2}\right)$. For any unconditional convex body $K$ (not necessarily smooth or capillary) with the standard support function $h_K$ we have
\begin{equation}
\frac{\vol(K)}{2n} 
\int_{\mathbb{S}^{n-1}_\theta} 
\frac{(1 - \cos\theta \, x_n)^{n+1}}{h_K^n(x)} \, d\si(x) 
\leq \vol(\widehat{\mathcal{C}_\theta})^2.
\end{equation}
In particular, when $\Sigma$ is an unconditional, smooth, strictly convex capillary hypersurface, we have
\begin{equation}
\vol(\widehat{\Sigma}) 
\vol(\widehat{\Sigma^*}) 
\leq \vol(\widehat{\mathcal{C}_\theta})^2.
\end{equation}
Moreover, equality holds if and only if 
$\Sigma = \lambda \mathcal{C}_\theta$ for some $\lambda > 0$.
\end{thm}

See also \autoref{cor 1}. Let $\theta \in \left(0, \frac{\pi}{2}\right)$ and $\Delta_{\mathcal{C}_\theta}$ denote the centro-affine Laplacian of $\mathbb{S}^{n-1} - \cos\theta \,E_n$ restricted to $\mathcal{C}_\theta$. By the previous theorem, $\mathcal{C}_\theta$ is a local maximizer of the volume product in the class of unconditional capillary bodies. By linearizing the inequality, we obtain the following result:

\begin{thm}\label{thm 2}
Let $\theta \in \left(0, \frac{\pi}{2}\right)$ and let $f \in C^2(\mathcal{C}_\theta)$ be an unconditional function satisfying $\bar{\nabla}_\mu f = 0$, where $\mu$ denotes the unit normal vector to $\partial \mathcal{C}_\theta \subset \mathcal{C}_\theta$ pointing towards the lower half-space. Then
\begin{equation}
\int_{\mathcal{C}_\theta} f \left( \Delta_{\mathcal{C}_\theta} f + 2n f \right) \, dV_{\mathcal{C}_\theta} \leq 2n \frac{\left( \int_{\mathcal{C}_\theta} f \, dV_{\mathcal{C}_\theta} \right)^2}{\int_{\mathcal{C}_\theta} dV_{\mathcal{C}_\theta}}.
\end{equation}
\end{thm}

\section{\texorpdfstring{Case $\theta\in (0,\frac{\pi}{2})$: Unconditional bodies}{Case theta in (0, pi/2}}
\subsection{The volume product}
Throughout this section we assume $\theta\in (0,\frac{\pi}{2})$. Let $C$ denote the strictly convex body obtained via $\wh{\mathcal{C}_{\theta}}$ and its reflection across $x_n=0$. The (standard) support function of $C$ at $x=(x',x_n)\in \bbR^n$ is given by
\eq{
h_C(x) =
\begin{cases}
|x'| \sin \theta & \text{if } |x_n| \leq |x| \cos \theta ,\\
|x| - |x_n| \cos \theta & \text{if } |x_n| \geq |x| \cos\theta.
\end{cases}
}
The Minkowski functional/gauge function of $C$ is given by
\eq{\label{def: C-Minkowski functional}
p_C(x)=\frac{\cos \theta \,|x_n|+\sqrt{|x_n|^2+\sin^2 \theta \,|x'|^2}}{\sin^2 \theta}.
}
This is the support function of a convex body whose boundary consists of a cylindrical part capped by two spheroids.

To obtain the formula for $p_C$, note that
\eq{
C=\left\{x=(x',x_n)\in \bbR^n:\ |x'|^2+\(|x_n|+\cos\theta\)^2\le 1\right\}.
}
By definition of the Minkowski functional, $p_C(x)$ is the unique $\la>0$ such that $x/\la\in C$. Equivalently,
\eq{
\frac{|x'|^2}{\la^2}+\left(\frac{|x_n|}{\la}+\cos\theta\right)^2=1.
}
Hence, $\la$ satisfies
\eq{
\sin^2\theta\,\la^2-2\cos\theta\,|x_n|\,\la-\(|x_n|^2+|x'|^2\)=0.
}
Finding the positive root, we obtain \eqref{def: C-Minkowski functional}.

Let $V=\frac{1}{2}p_C^2$ and $V^{\ast}=\frac{1}{2}h_C^2$. Note that $V^{\ast}$ is the Legendre transform of $V$; see \cite[eq. (1.49)]{Schneider} and \cite[Lem. 1.7.13]{Schneider}.

For $x,y\in [0,\infty)^n$, let us adopt the notation
\eq{
\sqrt{xy}=(\sqrt{x_1y_1},\ldots, \sqrt{x_ny_n}).
}

\begin{lemma}\label{lem 1} The function
 $x\mapsto V(\sqrt{x})$ is concave in $(0,\infty)^n$.
\end{lemma}
\begin{proof}
Let $s=x_1+\cdots+x_{n-1}$ and $x_n=t$. Define
\eq{
v(s,t)=V(\sqrt{x})=\frac{1}{2\sin^4\theta}\left((1+\cos^2\theta)t+(\sin^2\theta)s +2\cos\theta \sqrt{t^2+s t\sin^2\theta}\right).
}
Then the Hessian of $v$ (a function of two variables $s,t$) is
\eq{
D^2v=\begin{bmatrix}
v_{ss} & v_{st}\\
v_{ts} & v_{tt}
\end{bmatrix}=-\frac{\cos\theta}{4(t^2+s t\sin^2\theta)^\frac{3}{2}}\begin{bmatrix}
    t^2 & -st \\
    -st & s^2
\end{bmatrix}.
}
From this we obtain
\begin{equation}
D^2V(\sqrt{x}) = -\frac{\cos\theta}{4(t^2+s t\sin^2\theta)^\frac{3}{2}}
\begin{bmatrix}
t^2 & \cdots & t^2 & -st \\
\vdots & \ddots & \vdots & \vdots \\
t^2 & \cdots & t^2 & -st \\
-st& \cdots & -st & s^2
\end{bmatrix}.
\end{equation}
Recall $\theta\in (0,\frac{\pi}{2})$. Then for any $a=(a_1,\ldots, a_n)\in \mathbb{R}^n$, there holds
\eq{
D^2V(\sqrt{x})(a,a)=-\frac{\cos\theta}{4(t^2+s t\sin^2\theta)^{\frac{3}{2}}}\big(t\sum_{i=1}^{n-1}a_i-sa_n\big)^2 \leq 0.
}
\end{proof}

Let $\Omega\subset \bbR^n$ be an open set, and let $u:\Omega\to \bbR$ be a convex function. For a Borel set $E\subset \Omega$, the Monge--Amp\`ere measure of $u$ on $E$ is defined by
\eq{
\mu_u(E):=|\partial u(E)|,
}
where $\partial u(E):=\bigcup_{x\in E}\partial u(x)$, and $\partial u(x)$ denotes the subdifferential of $u$ at $x$, namely
\eq{
\partial u(x):=\{p\in \bbR^n:\; u(y)\geq u(x)+\ip{p}{y-x}\ \text{for all } y\in \Omega\}.
}
Here $|\cdot|$ denotes the Lebesgue measure in $\bbR^n$.

\begin{lemma}\label{lem 2}
The function $V^{\ast}$ is of class $C^1$, and for $y=(y',y_n)$ satisfying $|y_n| \leq |y| \cos \theta$, we have
\begin{equation}
DV^{\ast}(y) = \sin^2 \theta \, (y', 0).
\end{equation}
Moreover, we have
\eq{
\det D^2V^{\ast}(y) =
\begin{cases}
0 & \text{if } |y_n| < |y| \cos \theta , \\
(1-\frac{|y_n|}{|y|}\cos\theta)^{n+1} & \text{if } |y_n| > |y| \cos \theta.
\end{cases}
}
In particular, the Monge-Amp\`ere measure of $V^{\ast}$ is
\eq{
\mu_{V^{\ast}}(A)=\int_{A\cap\{y:\, |y_n|>|y|\cos\theta \}}\(1-\frac{|y_n|}{|y|}\cos\theta\)^{n+1}\, dy,
}
for any Borel set $A$.
\end{lemma}
\begin{proof} 
Since $C$ is strictly convex, by \cite[Cor. 1.7.3]{Schneider}, $h_C$ is differentiable at $y \neq 0$. Since $V^{\ast}$ is two-homogeneous, we have $DV^{\ast}(0) = 0$. Moreover, a convex differentiable function on $\mathbb{R}^n$ is $C^1$; see \cite[Thm. 1.5.4]{Schneider}.

The second claim follows from a well-known identity: a smooth, strictly convex hypersurface with support function $h$ and Gauss curvature $\cK$ satisfies
\eq{
\det D^2\frac{1}{2}h^2(x)=\frac{h^{n+1}}{\mathcal{K}}\left(\frac{x}{|x|}\right).
}
Here the Gauss curvature is regarded as a function on the unit sphere via the Gauss map. See, for instance, \cite[Thm. 3.1]{MP14}.
\end{proof}

\begin{lemma}\label{lem 3}Let $A=(0,\infty)^n$ and $B=\{ y \in (0,\infty)^n :\, y_n> |y|\cos\theta \}$. Then $ D V(A) =B$, $DV: A\to B$ is a $C^{\infty}$ diffeomorphism and
\eq{
\det D^2V(x)\det D^2V^{\ast}(DV(x))=1\q \forall x\in A.
}
\end{lemma}
\begin{proof}
Let $r=\sqrt{x_n^2+|x'|^2\sin^2\theta}$ and define $\phi\in (0,\frac{\pi}{2})$ through $\cos \phi=\frac{x_n}{r}$ and $\sin \phi=\frac{\sin\theta |x'|}{r}$. Let $y=(y',y_n)=DV(x)$. Then
\eq{
y_n&=\frac{r}{\sin^4\theta}(1+\cos\theta \cos\phi)(\cos\theta+\cos\phi),\\
y'&=\frac{1+\cos\theta\cos\phi}{\sin^2\theta}x',\quad |y'|=\frac{r}{\sin^3\theta}(1+\cos\theta\cos\phi)\sin\phi.
}

Define $f(\phi)=\frac{\sin\theta \sin\phi}{\cos\theta+\cos\phi}=\frac{|y'|}{y_n}$ for $\phi\in (0,\frac{\pi}{2})$. Since $f$ is strictly increasing and maps $(0,\frac{\pi}{2})$ onto $(0,\tan\theta)$, we have $0<|y'|/y_n<\tan\theta$, which is equivalent to $y_n>|y|\cos\theta$. Hence $DV(A)\subset B$.

Let $y\in B$. We choose $\phi\in(0,\frac{\pi}{2})$ such that $f(\phi)=|y'|/y_n$, and define
\eq{
r=\frac{\sin^4\theta\, y_n}{(1+\cos\theta\cos\phi)(\cos\theta+\cos\phi)}.
}
Then
\eq{
x=r\left(\frac{\sin\phi}{\sin\theta}\frac{y'}{|y'|},\cos\phi\right)\in A \quad \text{and} \quad  DV(x)=y.
}
Hence $B\subset DV(A)$.

Note that $V=\frac{1}{2}p_C^2$ is $C^\infty$-smooth and strictly convex on the open set $A$, and thus the gradient map $DV: A \to B$ is bijective. By the inverse function theorem, since $D^2 V(x)$ is positive definite for $x \in A$, the map $DV$ is a diffeomorphism. Moreover, from the definition of the Legendre transform, for $y \in B$, we have $DV^{\ast}(y) = x$, where $x \in A$ satisfies $y = DV(x)$. Thus, $DV^{\ast}(DV(x))=x$ and the relation between the determinants follows.
\end{proof}

We next recall the Pr\'ekopa--Leindler inequality, along with the characterization of its equality cases; see also \cite{BB10, BD21, BFR23}.

\begin{theorem}[Pr\'ekopa--Leindler inequality, equality by Dubuc]\cite{Pre71,Lei72,Bor75,Dub77}
Let $\lambda \in (0, 1)$ and let $f$, $g$ and $h$ be non-negative integrable functions on $\mathbb{R}^n$. Assume that
\[
h((1 - \lambda)x + \lambda y) \geq f(x)^{1-\lambda} g(y)^{\lambda}
\]
for all $x, y \in \mathbb{R}^n$. Then
\[
\int_{\mathbb{R}^n} h  \geq \left( \int_{\mathbb{R}^n} f \right)^{1-\lambda} \left( \int_{\mathbb{R}^n} g  \right)^{\lambda}.
\]
If $f$ and $g$ have positive integrals and equality holds, then there exist $a > 0$, $w \in \mathbb{R}^n$, and a log-concave function $\tilde{h}$, such that $h = \tilde{h}$, $f(x) = a^{\la} \tilde{h}(x +\la w)$, and $g(x) = a^{ -(1-\lambda)} \tilde{h}(x -(1- \lambda) w)$ for almost every $x$.
\end{theorem}

Note that $V$ is unconditional, two-homogeneous, and $C^{\infty}$-smooth on $(0,\infty)^n$. 
By \autoref{lem 1} and \cite[Lem. 5.17]{CKLR24}, we have
\begin{equation}\label{2-concavity}
    \langle a, DV(b) \rangle \;\geq\; \langle \sqrt{ab}, DV(\sqrt{ab}) \rangle 
    \;=\; 2 V(\sqrt{ab}) 
    \quad \forall\, a,b \in (0,\infty)^n.
\end{equation}

As in the proof of \cite[Thm. 5.19]{CKLR24} (see also \cite[Prop. 1]{FM07}), 
for any unconditional convex function $\Phi$ we have
\begin{equation}\label{key-ineq}
    \int_{(0,\infty)^n} e^{-\Phi(x)}\,dx 
    \int_{(0,\infty)^n} e^{-\Phi^{\ast}(DV(x))}\,dx 
    \;\leq\;
    \left( \int_{(0,\infty)^n} e^{-V(x)}\,dx \right)^2.
\end{equation}
To see this, define $f,g,h : \mathbb{R}^n \to (0,\infty)$ by
\begin{equation*}
    f(t) = e^{-\Phi(e^t) + \sum_{i=1}^n t_i},\, g(s) = e^{-\Phi^{\ast}(DV(e^s)) + \sum_{i=1}^n s_i},\, h(r) = e^{-V(e^r) + \sum_{i=1}^n r_i},
\end{equation*}
where $e^t := (e^{t_1}, \ldots, e^{t_n})$ for any $n$-tuple $t=(t_1,\ldots,t_n) \in \mathbb{R}^n$. From the definition of the Legendre transform together with \eqref{2-concavity}, 
it follows that
\eq{
\Phi(e^t)+\Phi^{\ast}(DV(e^s))\geq \langle e^t,DV(e^s)\rangle\geq  2V(e^{\frac{t+s}{2}}).
}
Hence, $\sqrt{f(t) g(s)} \;\leq\; h\!\left(\tfrac{s+t}{2}\right)$, and by the Pr\'ekopa--Leindler inequality,
\begin{equation*}
    \int_{\mathbb{R}^n} f \,\int_{\mathbb{R}^n} g 
    \leq
    \left( \int_{\mathbb{R}^n} h \right)^2.
\end{equation*}
Finally, note the following identities:
\begin{align*}
    \int_{(0,\infty)^n} e^{-\Phi(x)}\,dx 
    &= \int_{\mathbb{R}^n} e^{-\Phi(e^t) + \sum_{i=1}^n t_i}\,dt, \\
    \int_{(0,\infty)^n} e^{-\Phi^{\ast}(DV(x))}\,dx 
    &= \int_{\mathbb{R}^n} e^{-\Phi^{\ast}(DV(e^s)) + \sum_{i=1}^n s_i}\,ds, \\
    \int_{(0,\infty)^n} e^{-V(x)}\,dx 
    &= \int_{\mathbb{R}^n} e^{-V(e^r) + \sum_{i=1}^n r_i}\,dr,
\end{align*}
which together yield \eqref{key-ineq}.

Suppose $K$ is an unconditional convex body. We take $\Phi = \frac{1}{2} p_K^2$, where $p_K$ denotes the Minkowski functional of $K$. Note that $\Phi$ is an unconditional convex function. Now we use the change of variables $y=DV(x)$. By \autoref{lem 2} and \autoref{lem 3}, we have
\eq{
\int_{(0,\infty)^n} e^{-\Phi^{\ast}(DV(x))}\, dx &= \int_{DV((0,\infty)^n)} e^{-\frac{1}{2} p_{K^{\circ}}^2(y)} \det D^2 V^{\ast}(y)\, dy\\
&= \int_{(0,\infty)^n \cap \{ y :\, y_n > |y| \cos \theta \}} e^{-\frac{1}{2} p_{K^{\circ}}^2(y)} \left( 1 - \frac{y_n}{|y|} \cos \theta \right)^{n+1} \, dy.
}

Using spherical coordinates and that $K$ and $C$ are unconditional, we have
\eq{
\int_{(0,\infty)^n \cap \{ y :\, y_n > |y| \cos \theta \}} e^{-\frac{1}{2} p_{K^{\circ}}^2(y)} \left( 1 - \frac{y_n}{|y|} \cos \theta \right)^{n+1} \, dy &= \frac{c}{2^{n-1}} \int_{\mathbb{S}_{\theta}^{n-1}} \frac{h_C^{n+1}}{h_K^n}\, d\si, \\
\int_{(0,\infty)^n} e^{-\frac{1}{2} p_K^2(x)}\, dx &= \frac{n c}{2^n} \vol(K), \\
\int_{(0,\infty)^n} e^{-\frac{1}{2} p_C^2(x)}\, dx &= \frac{n c}{2^{n-1}} \vol(\widehat{\mathcal{C}_{\theta}}),
}
where $c := \int_0^{\infty} e^{-\frac{1}{2} s^2} s^{n-1} ds$; see also \cite[eq. (1.54)]{Schneider}. Therefore,
\eq{\label{ineq x}
\frac{\vol(K)}{2n} \int_{\mathbb{S}_{\theta}^{n-1}} \left( \frac{h_C}{h_K} \right)^n h_C\, d\si \leq \vol(\widehat{\mathcal{C}_{\theta}})^2.
}
This completes the proof of the first inequality in \autoref{thm 1}.

To prove the second inequality in \autoref{thm 1}, take $K = \widehat{\Sigma} \cup R(\widehat{\Sigma})$, where $\Sigma$ is an unconditional, strictly convex capillary hypersurface and $R$ denotes the reflection map across $x_n = 0$. Since
\eq{
\vol(\widehat{\Sigma^{\ast}}) = \frac{1}{n} \int_{\mathcal{C}_{\theta}} \frac{\ell^{n+1}}{s_{\Sigma}^n}\, d\si = \frac{1}{n} \int_{\mathbb{S}_{\theta}^{n-1}} \left( \frac{h_C}{h_K} \right)^n h_C\, d\si,
}
we obtain
\eq{
\vol(\widehat{\Sigma}) \vol(\widehat{\Sigma^{\ast}}) \leq \vol(\widehat{\mathcal{C}_{\theta}})^2.
}

Next, we proceed with the characterization of the equality cases. By the characterization of the equality cases  of the Pr\'ekopa--Leindler inequality, for some $a>0$ and $w\in \bbR^n$, we have 
\eq{
\Phi(e^t)=-\frac{1}{2}\log a+V(e^{t+\frac{1}{2}w})-\frac{1}{2}\sum_{i=1}^nw_i\q \forall t\in \bbR^n.
}
Sending all $t_i\to -\infty$, we obtain
\eq{
0=\log a+\sum_{i=1}^nw_i.
}
That is,
\eq{
p_K(e^t)=p_C(e^{t+\frac{1}{2}w})\q \forall t\in \bbR^n.
}
Therefore, for the matrix $A:=\operatorname{diag}(e^{\frac{1}{2}w_1},\ldots,e^{\frac{1}{2}w_n})$, we have
\eq{
p_{AK}(x)=p_C(x)\q \forall x\in (0,\infty)^n,
}
where we used $p_K(A^{-1}x)=p_{AK}(x)$. Since $AK,C$ are both unconditional, we must have $A\Si=\cC_{\theta}$. Due to  the following lemma, $w_1=\cdots=w_n$ and so 
\eq{
\Si=e^{-\frac{1}{2}w_n}\cC_{\theta}.
}

\begin{lemma}
	Let $A=\operatorname{diag}(a_1,\ldots,a_n)$ with $a_i>0$. If $A\cC_{\theta}$ is $\theta$-capillary, then for all $i$ we have $a_i=a_n$.
\end{lemma}
\begin{proof}
Let $x \in \mathbb{S}^{n-1}_{\theta}$. The outer unit  normal to $A \mathbb{S}^{n-1}_{\theta}$ at $Ax$ is given by
\eq{
    \nu(x) = \frac{A^{-1}x}{|A^{-1}x|},
}
and we have
\eq{
    \langle \nu(x), E_n \rangle 
    = \frac{x_n}{a_n\sqrt{\sum_{i=1}^n \tfrac{x_i^2}{a_i^2}}}.
}
In particular, for $i=1,\ldots,n-1$, the point $e_i = \sin\theta\, E_i + \cos\theta\, E_n \in \partial \mathbb{S}^{n-1}_{\theta}$ satisfies
\eq{
    \langle \nu(e_i), E_n \rangle
    = \frac{\cos\theta}{\sqrt{\cos^2\theta + \tfrac{a_n^2}{a_i^2}\sin^2\theta}}.
}
If $A \mathcal{C}_{\theta}$ is $\theta$-capillary, then there holds
\eq{
    \cos^2\theta + \frac{a_n^2}{a_i^2}\sin^2\theta = 1, \quad i=1,\ldots,n-1.
}
Hence, $a_i = a_n$ for all $i$.
\end{proof}

\begin{cor}\label{cor 1}Let $\alpha \in (-n, 0)$. For an unconditional convex body $K$ define
\eq{
\mathcal{E}_\alpha(K) = \int_{\mathbb{S}_{\theta}^{n-1}} \left( \frac{h_K}{h_C} \right)^\alpha h_C\, d\si.
}
If $\vol(K) = \vol(C)$, then
\eq{
\mathcal{E}_\alpha(K) \leq \mathcal{E}_\alpha(C) &= \int_{\mathbb{S}_{\theta}^{n-1}} h_C\, d\si, \\
\int_{\mathbb{S}_{\theta}^{n-1}} h_C \log h_K\, d\si &\geq \int_{\mathbb{S}_{\theta}^{n-1}} h_C \log h_C\, d\si.
}
Moreover, equality holds only when $K = C$.
\end{cor}
\begin{proof}
The inequalities follow from Jensen's inequality and \eqref{ineq x}. If equality holds in either inequality, then equality must also hold in Jensen's inequality, and therefore $h_K/h_C$ is constant on $\mathbb{S}^{n-1}_\theta$ and hence $K=C$.
\end{proof}

\subsection{Linearization}
We linearize \eqref{conj} at $\cC_{\theta}$. Let $f:\mathcal C_\theta\to\bbR$ be an even $C^2$ function satisfying the Neumann condition $\bar\nabla_\mu f=0$, and set $\psi=f\ell$. Since $\bar\nabla_\mu\ell=\cot\theta\,\ell$ on $\partial\mathcal C_\theta$, it follows that
\eq{
\bar\nabla_\mu\psi=\cot\theta\,\psi.
}
Moreover, for sufficiently small $|t|$, the function $\ell+t\psi$ remains a capillary support function of a strictly convex capillary hypersurface; see \cite[Prop. 2.6]{MWWX25}.

Define
\eq{
\cP(t) = \vol(\wh{\Si_t}) \vol(\wh{\Si_t^{\ast}}).
}
Then, integrating by parts, we obtain
\begin{align}\label{a}
    n \frac{d}{dt} \cP = \int_{\cC_{\theta}} \frac{\psi}{G_{\Si_t}}\int_{\cC_{\theta}} \frac{\ell^{n+1}}{s_{\Si_t}^{n}} - \int_{\cC_{\theta}} \frac{s_{\Si_t}}{G_{\Si_t}} \int_{\cC_{\theta}} \frac{\ell^{n+1}}{s_{\Si_t}^{n+1}} \psi
\end{align}
and
\eq{
n \frac{d^2}{dt^2} \Big|_{t=0} \cP &= \int_{\cC_{\theta}} \psi (\bar{\De} \psi + (n-1) \psi) \int_{\cC_{\theta}} \ell - 2n \left( \int_{\cC_{\theta}} \psi \right)^2 + (n+1) \int_{\cC_{\theta}} \ell \int_{\cC_{\theta}} \frac{\psi^2}{\ell},
}
where all integrals are taken with respect to $d\si$. 

Rewriting the second derivative of $\cP$ at $t=0$ in terms of the centro-affine geometry of $\cC_{\theta}$ (more precisely, of $\bbS^{n-1} - \cos\theta\, E_{n}$), for any even function $f \cn \cC_{\theta} \to \bbR$ with $\bar{\nabla}_{\mu} f = 0$, we obtain
\eq{
\frac{1}{n} \frac{d^2}{dt^2} \Big|_{t=0} \cP &= \int_{\cC_{\theta}} f (\De_{\cC_{\theta}} f + (n-1) f) \, dV_{\cC_{\theta}} \int_{\cC_{\theta}} dV_{\cC_{\theta}} \\
&\q - 2n \left( \int_{\cC_{\theta}} f \, dV_{\cC_{\theta}} \right)^2 + (n+1) \int_{\cC_{\theta}} dV_{\cC_{\theta}} \int_{\cC_{\theta}} f^2 \, dV_{\cC_{\theta}}.
}
For detailed background on centro-affine geometry, we refer the reader to \cite{Mil23}. In the present argument, we only use the identity
\eq{
\Delta_{\mathcal{C}_{\theta}} f + (n-1) f
=\bar{\Delta} \psi + (n-1)\psi,
}
where $\bar{\Delta}$ denotes the Laplace--Beltrami operator on $(\bbS^{n-1},\bar g,\bar \nabla)$.

By \autoref{thm 1}, for $\theta\in (0,\tfrac{\pi}{2})$,  $\cC_{\theta}$ is a local maximizer of $\vol(\wh{\Si}) \vol(\wh{\Si^{\ast}})$ in the class of unconditional capillary hypersurfaces. Hence, if $f\in C^2(\cC_{\theta})$ is an unconditional function with $\bar{\nabla}_{\mu}f=0$, then 
\eq{\label{linearized-ineq}
\int_{\cC_{\theta}} f (\De_{\cC_{\theta}} f + 2n f) \, dV_{\cC_{\theta}} \leq 2n \frac{\left( \int_{\cC_{\theta}} f \, dV_{\cC_{\theta}} \right)^2}{\int_{\cC_{\theta}} dV_{\cC_{\theta}}}.
}
This completes the proof of \autoref{thm 2}.

\section{\texorpdfstring{Case $\theta\in (\frac{\pi}{2}, \pi)$: Unboundedness}{Case theta in (pi/2, pi)}}
From the proof of \autoref{lem 1}, it follows that $V(\sqrt{x})$ is convex for $\theta \in (\frac{\pi}{2}, \pi)$, which makes the method used in the proof of \autoref{thm 1} ineffective. In this section, we show that the volume product is in fact unbounded when $\theta \in (\frac{\pi}{2}, \pi)$.

\textbf{Example 1:} We construct a family of $C^{1,1}$-smooth convex bodies that are smooth near $x_{n} = 0$ and meet this hyperplane at the constant angle $\theta \in (\frac{\pi}{2}, \pi)$.  In this example, we write points in $\bbR^n$ in the form
$(x,y)$ with $x\in\bbR^{n-1}$ and $y\in\bbR$. For $\la > 0$, define
\[
\begin{aligned}
K_{\la} &\coloneqq \left\{ (x, y) \in \frac{1}{\la} \wh{\cC_\theta} : y \leq -\frac{\cos \theta}{\la} \right\} \\
&\q \cup \left\{ (x, y) \in \bbR^n : |x| \leq \frac{1}{\la}, -\frac{\cos \theta}{\la} \leq y \leq \la^{n-1} \right\} \cup \left( \frac{1}{\la} B^n + \la^{n-1} E_n \right).
\end{aligned}
\]

Note that $K_{\lambda}$ consists of a portion of capillary cap, joined to a vertical cylinder of radius $1/\lambda$, and topped by a ball of radius $1/\lambda$ whose center is translated upward by $\lambda^{n-1} E_n$. 

We have $\vol(\la K_{\la}) \geq \omega_{n-1} (\la^n + \cos \theta)$ and thus $\lim_{\la \to \8} \vol(K_{\la}) \geq \omega_{n-1}$ where $\omega_{n-1}$ is the volume of the unit ball in $\mathbb{R}^{n-1}$. 

The capillary support function $s_{\Si_\la}$ of the boundary hypersurface $\Si_\la=\overline{\partial K_{\la}\cap \bbR^n_+}$, restricted to the set $\cD_{\theta} \coloneqq \cC_\theta \cap \left\{ (x, y) \in \bbR^n : y \leq -\frac{1}{2}\cos \theta \right\}$, is:
\eq{
s_{\Si_\la}(\ze) = \ip{\frac{1}{\la} \ze}{\ze + \cos \theta E_n} \leq \frac{1 - \cos \theta}{\la}.
}
Moreover, we have
\eq{
n\lim_{\la \to \8} \vol(\wh{\Si_\la^*}) = \lim_{\la \to \8} \int_{\cC_\theta} \frac{\ell^{n+1}}{s_{\Si_\la}^n} \, d\si\geq \lim_{\la \to \8} \int_{\cD_{\theta}} \frac{\ell^{n+1}}{s_{\Si_\la}^n} \, d\si = \8,
}
meaning that $\vol(\wh{\Si_\la^*}) \to \8$, and thus $\vol(\wh{\Si_\la})\vol(\wh{\Si_\la^*}) \to \8$ as $\la \to \8$. 

\textbf{Example 2:} We consider the intersection of an ellipse and the half-space $y\geq 0$ given by:
\eq{
K= \left\{(x,y): \frac{x^2}{a^2} + \frac{(y - c)^2}{b^2} \leq 1, \q y \geq 0\right\},
}
where $a > 0$, $b > c > 0$ are constants to be determined later. Let us put $\Si=\overline{\partial K\cap \bbR^2_+}$. Suppose the contact angle of $\Si$ with $y=0$ is $\theta\in (\frac{\pi}{2},\pi)$. 

Using the polar coordinates, $\Si$ can be expressed by
\eq{
X(t)=(a\cos t, b\sin t+c),\quad  -\arcsin(\frac{c}{b})\leq t \leq \pi+\arcsin(\frac{c}{b}).
}
The unit outward normal at the point  $X(t)$ is given by
\eq{
N(t)=\frac{(b\cos t,a\sin t)}{\sqrt{b^2\cos^2 t+a^2\sin^2 t}}.
}
Note that the range of $N(t)$ as $t$ varies is the set
\eq{
\bbS_{\theta}^1:=\left\{(x,y)\in \bbS^1: y\geq \cos\theta=-\frac{ac}{\sqrt{b^4+(a^2-b^2)c^2}}\right\}.
}

We take $a=1/b$ and set $c = \eta b$, where $\eta \in (0, 1)$. Then
\eq{
\vol(\wh{\Si}) \geq \frac{\pi}{2}.
}
Moreover, we have
\eq{
\cos \theta &=- \frac{ac}{\sqrt{b^4+(a^2-b^2)c^2}}
&= -\frac{\eta}{\rt{(1 - \eta^2)b^4 + \eta^2}}
}
and hence
\eq{
\eta = -\frac{b^2 \cos\theta}{\rt{b^4 \cos^2 \theta + \sin^2 \theta}} = \frac{b^2}{\rt{b^4 + \tan^2 \theta}}.
}

The support function of $\wh{\Si}$, $h_{\wh{\Si}}$, over $\bbS_{\theta}^1$ is given by 
\eq{
h_{\wh{\Si}}(\psi)=h_{\wh{\Si}}(\cos \psi,\sin \psi) &=c \sin \psi+\sqrt{a^2 \cos^2 \psi+b^2\sin^2 \psi}\\
&=\frac{b^3 \sin \psi}{\sqrt{b^4+\tan^2\theta}}+\sqrt{b^{-2}\cos^2 \psi+b^2\sin^2 \psi},
}
where $\psi \in [\frac{\pi}{2}-\theta,\frac{\pi}{2}+\theta]$.
Then, for some constant $c_{\theta}>0$ we get
\eq{
\vol(\wh{\Si^{\ast}}) &\geq \frac{c_{\theta}}{2}\int_{\frac{\pi}{2}-\theta}^{\frac{\pi}{2}+\theta} \frac{1}{h_{\wh{\Si}}^2(\psi)} \, d\psi \\
&=c_{\theta} \int_{\frac{\pi}{2}-\theta}^{\frac{\pi}{2}} \frac{1}{(\frac{b^3 \sin \psi}{\sqrt{b^4+\tan^2\theta}}+\sqrt{b^{-2}\cos^2 \psi+b^2\sin^2 \psi})^2} \,d\psi.
}

Assume that $b\gg1$ and $\psi\in (\frac{\pi}{2}-\theta,\frac{\pi}{2}-\theta+\de)\subset (-\frac{\pi}{2},0)$, where $\de:=b^{-3}$. Then we have
\eq{\label{eq: ineq}
\sqrt{b^{-2} \cos^2 \psi + b^2 \sin^2 \psi} \leq \sqrt{b^{-2} \sin^2 \theta + b^2 \cos^2 \theta}.
}
To see this, let us define the function
\[
f(\alpha) = \sqrt{b^{-2} \cos^2 \alpha + b^2 \sin^2 \alpha},\q  \alpha \in ( -\frac{\pi}{2}, 0 ).
\]
Let $ \phi = \frac{\pi}{2} - \theta $. Since $ \theta \in \left( \frac{\pi}{2}, \pi \right) $, we have $ \phi \in \left( -\frac{\pi}{2}, 0 \right) $, and $ \psi \in (\phi, \phi + \delta) \subset \left( -\frac{\pi}{2}, 0 \right) $. The right-hand side of the inequality \eqref{eq: ineq} is
\[
\sqrt{b^{-2} \sin^2 \theta + b^2 \cos^2 \theta} = \sqrt{b^{-2} \cos^2 \phi + b^2 \sin^2 \phi} = f(\phi).
\]
Thus, the inequality reduces to $ f(\psi) \leq f(\phi) $ for $\psi \in (\phi, \phi + \delta)$. Define
\[
g(\alpha) = f(\alpha)^2 = b^{-2} \cos^2 \alpha + b^2 \sin^2 \alpha.
\]
Then
\[
g'(\alpha)  = -2 b^{-2} \cos \alpha \sin \alpha + 2 b^2 \sin \alpha \cos \alpha =   (b^2 - b^{-2})\sin 2\al<0.
\]
Thus $ f(\alpha) $ is decreasing.

Now using inequality \eqref{eq: ineq} we obtain
\eq{
h_{\wh{\Si}}(\psi)= &~\frac{b^3 \sin \psi}{\sqrt{b^4+\tan^2\theta}}+\sqrt{b^{-2}\cos^2 \psi+b^2\sin^2 \psi}\\
\leq &~\frac{b^3 \sin(\phi+\de)}{\sqrt{b^4+\tan^2\theta}}+\sqrt{b^{-2}\sin^2 \theta+b^2\cos^2 \theta}\\
=&~\frac{b^3 \cos(\theta-\de)}{\sqrt{b^4+\tan^2\theta}}+b^{-1}(-\cos\theta)\sqrt{b^4+\tan^2 \theta} \\
=&~\frac{b^3(\cos\theta\cos\de+\sin\theta \sin\de)-b^3\cos\theta-b^{-1}\sin\theta\tan\theta}{\sqrt{b^4+\tan^2\theta}}\\
= &~\frac{b^3(\cos\theta \,(\cos\de-1)+\sin\theta\sin\de)-b^{-1}\sin\theta\tan\theta}{\sqrt{b^4+\tan^2\theta}}.
}
Note that the leading term as $b\to \infty$ is
\eq{
\frac{b^3(-\frac{\de^2}{2}\cos\theta+\de\sin\theta )-b^{-1}\sin\theta\tan\theta}{b^2}\sim (\sin\theta) b^{-2}.
}
Therefore, we obtain
\eq{
\vol(\wh{\Si^{\ast}})\geq &~ c_{\theta}'\int_{\frac{\pi}{2}-\theta}^{\frac{\pi}{2}-\theta+\de}\frac{ b^{4}}{\sin^2\theta} \, d\psi \\
 =&~ \frac{c_{\theta}' b}{\sin^2\theta}  \ra \infty, \quad \text{as $b \ra \infty$}.
}
That is, $\vol(\wh{\Si}) \vol(\wh{\Si^{\ast}})  \ra \infty$ as $b \ra \infty$.

\section*{Acknowledgments}

The authors thank the referee for valuable suggestions.

Cabezas-Moreno and Ivaki were supported by the Austrian Science Fund (FWF) under Project P36545. Hu was supported by the National Key Research and Development Program of China (Grant No. 2021YFA1001800).

\vspace{25mm}

	\textsc{Institut f\"{u}r Diskrete Mathematik und Geometrie,\\ Technische Universit\"{a}t Wien,\\ Wiedner Hauptstra{\ss}e 8-10, 1040 Wien, Austria,\\} 
\email{\href{mailto:carlos.moreno@tuwien.ac.at}{carlos.moreno@tuwien.ac.at}}

\vspace{5mm}
	
	\textsc{School of Mathematical Sciences, Beihang University,\\ Beijing 100191, China,\\}
	\email{\href{mailto:huyingxiang@buaa.edu.cn}{huyingxiang@buaa.edu.cn}}

\vspace{5mm}
        
\textsc{Institut f\"{u}r Diskrete Mathematik und Geometrie,\\ Technische Universit\"{a}t Wien,\\ Wiedner Hauptstra{\ss}e 8-10, 1040 Wien, Austria,\\} \email{\href{mailto:mohammad.ivaki@tuwien.ac.at}{mohammad.ivaki@tuwien.ac.at}}

\end{document}